\documentclass[12pt]{amsart}
\usepackage{amsmath, amsfonts, amssymb, amstext, amscd, amsthm, supertabular, hyperref, xypic, mathrsfs}
\usepackage{textcomp}
\usepackage[parfill]{parskip}

\usepackage[OT2,OT1]{fontenc}

\usepackage{color}
\newcommand{\QQ}{\mathbb Q}
\newcommand{\ZZ}{\mathbb Z}
\newcommand{\RR}{\mathbb R}

\newcommand{\tn}[1]{\textnormal{#1}}

\newcommand{\primorial}[1]{#1\textnormal{\textinterrobang}}
\newcommand{\abcs}{\frak X}
\newcommand{\insulator}[1]{{\mathcal I}(#1)}

\newtheorem*{main}{Main Theorem}
\newtheorem{theorem}{Theorem}
\newtheorem{lemma}{Lemma}

\newtheorem{proposition}[theorem]{Proposition}

\newtheorem{conj}{Conjecture}

\begin{document}
\title{Insulators of ABC Solutions}
\author{James E. Weigandt}
\begin{abstract} This note studies various ways of measuring the complexity of primitive solutions to the equation $A \,+\, B\, + \,C = 0$ and conjectures relating them. We define the insulator $\mathcal I(A,B,C)$ of a primitive solution as the smallest positive integer $\mathcal I$ such that the primes dividing the product $ABC \cdot \mathcal I$ are exactly those below a given bound. We show that the strong XYZ Conjecture of Lagarias and Soundararajan implies there are only finitely many primitive solutions $(A,B,C)$ with a given insulator.
\end{abstract}
\maketitle

\section{Introduction}
This note concerns integer solutions to the ternary linear equation
\begin{equation}\label{abceqn}
A  + B + C  = 0.
\end{equation}
We call such solutions {\bf primitive} if $\gcd(A,B,C) = 1$ and {\bf non-cuspidal} if the product $ABC$ is non-zero. Let $\abcs$ denote the set of all primitive and non-cuspidal integer solutions to (\ref{abceqn}). 

We measure the complexity of these solutions, among other ways, by the {\bf height} $H(A,B,C) = \max \{ |A|, |B|, |C| \}$, and the  {\bf smoothness}  $S(A,B,C)$ defined as the largest prime divisor of $ABC$. 

The Lagarias--Soundararajan XYZ Conjecture \cite{MR2780626} concerns the simultaneous behavior of $H(A,B,C)$ and $S(A,B,C)$. As the name suggests, the XYZ Conjecture is related to the celebrated ABC Conjecture, known to have far-reaching consequences. The main theorem of this note is a consequence of the XYZ Conjecture which seems surprising and does not appear to follow easily from the ABC Conjecture. 

Say that a non-zero integer is {\bf insulated} if it is divisible by exactly the primes less than a given bound and define the {\bf insulator} $\mathcal I(A,B,C)$ as the smallest positive integer $\mathcal I$ such that $ABC \cdot \mathcal I$ is insulated.
\begin{main} If the strong form of the XYZ Conjecture is true, then there are only finitely many $(A,B,C) \in \abcs$ with any fixed value of the insulator $\mathcal I(A,B,C)$.
\end{main}

\subsection{Acknowledgements} 

This material is based upon work supported by the National Science Foundation Graduate Research Fellowship under Grant No. DGE-1333468. The author thanks Edray Goins and Jeffery Lagarias for helpful conversations.

 \section{Semi-Heights on ABC Solutions}
Define a {\bf semi-height} on $\abcs$ as a function $h : \abcs \to \RR$ for which the set $\abcs_h(X) = \{ (A,B,C) \in \abcs \mid h(A,B,C) < X \}$
is finite for all real numbers $X > 0$. This section describes some semi-heights on $\abcs$.

For non-zero integers $n$, define $\tn{rad}(n)$ as the largest square-free divisor of $n$ and $P^+(n)$ as the largest non-composite divisor of $n$.

\begin{lemma} \label{semi_heights} Define the {\bf height} by $H(A,B,C) = \max\{|A|, |B|,|C|\}$, the {\bf conductor} by $N(A,B,C) = \tn{rad}(ABC)$, and the {\bf smoothness} by $S(A,B,C) = P^+(ABC)$. These are all semi-heights on $\abcs$.
\end{lemma}
\begin{proof} The functions $H$, $N$, and $S$ are related by the inequality
\begin{equation}
S(A,B,C) \leq \min \{ H(A,B,C), N(A,B,C) \}.
\end{equation}
It follows that $\abcs_H(X) \cup \abcs_N(X) \subseteq \abcs_S(X)$ for all $X > 0$ and it suffices to show that $\abcs_S(X)$ is finite for all $X > 0$. 

Given $X > 0$, define $\Sigma = \Sigma(X)$ as the finite set consisting of all primes less than $X$. If $S(A,B,C) < X$ then the associated solution $(x,y) = (-A/C,-B/C)$ to the unit equation
\begin{equation}\label{unit}
x + y = 1
\end{equation} 
must satisfy $x,y \in \ZZ_\Sigma^\times$. Since $\ZZ_{\Sigma}^\times$ is a finitely generated subgroup of $\bar{\QQ}^\times$, it follows from \cite[Theorem 5.2.1]{MR2216774} that $(x,y) = (-A/C,-B/C)$ belongs to a finite set $U(X)$ which depends on $X$. Since two solutions $(A,B,C), (A',B',C') \in \abcs$ satisfy $(-A/C, -B/C) = (-A'/C', -B'/C')$
if and only if $(A', B', C') = (uA, uB, uC)$ for some $u \in \ZZ^\times = \{\pm 1\}$, the finiteness of $\abcs_S(X)$ follows from that of $U(X)$.
\end{proof}

There are far reaching conjectures relating the above semi-heights. For example, the celebrated ABC Conjecture of Masser and Oesterl\'e relates $H(A,B,C)$ and $N(A,B,C)$.
\begin{conj}[ABC Conjecture, \cite{MR992208}] \label{abcconj} For every $\epsilon > 0$, there are only finitely many $(A,B,C) \in \abcs$ that satisfy $H(A,B,C)  > N(A,B,C)^{1 + \epsilon}$.
\end{conj}
The condition that $\epsilon > 0$ is necessary. To see this, note that for $k \geq 2$ we have  $(A_k, B_k, C_k) = (2^k(2^k -2), 1, -(2^k - 1)^2) \in \abcs$ and
\begin{equation}
H(A_k, B_k, C_k) =  (2^k - 1)^2  > (2^k - 1)(2^{k} - 2) \geq N(A_k, B_k, C_k).
\end{equation}
Recently, Lagarias and Soundararajan formulated a conjecture relating $H(A,B,C)$ and $S(A,B,C)$.
\begin{conj}[XYZ Conjecture (strong form), \cite{MR2780626}] \label{xyzconj} \tn{ } 
\begin{itemize}
	\item[(a)] For every $\epsilon > 0$, only finitely many $(A,B,C) \in \abcs$ satisfy 
	$$\log H(A,B,C) > S(A,B,C)^{2/3 + \epsilon}.$$
	\item[(b)] For every $\epsilon > 0$, infinitely many $(A,B,C) \in \abcs$ satisfy
	$$\log H(A,B,C) > S(A,B,C)^{2/3 - \epsilon}.$$
\end{itemize}
\end{conj}

The XYZ and ABC conjectures are related by \cite[Theorem 1.1]{MR2780626} which asserts that the ABC conjecture implies a weaker version of Conjecture \ref{xyzconj}(a) where the exponent $2/3 + \epsilon$ is replaced by $1 + \epsilon$. We state here a weak version of Conjecture \ref{xyzconj}(a) which we don't know to follow from Conjecture \ref{abcconj} at the time of this writing.

\begin{conj}\label{weak_xyz} For every $\delta > 0$ there are only finitely many solutions $(A,B,C) \in \abcs$ for which $\log H(A,B,C) > \delta \, S(A,B,C)$.
\end{conj} 

\begin{proposition}\label{prop1} Conjecture \ref{xyzconj}(a) implies Conjecture \ref{weak_xyz}.
\end{proposition} 
\begin{proof} Let $\delta > 0$ be given and let $X = \delta^{-6}$. Assume that Conjecture \ref{weak_xyz} is false. Assuming \ref{xyzconj}(a) we can then find infinitely many $(A,B,C) \in \abcs$ such that
\begin{equation}
	\delta S(A,B,C) < \log H(A,B,C) < S(A,B,C)^{5/6}
\end{equation}
If follows that $S(A,B,C) <  \delta^{-6} = X$ for infinitely many $(A,B,C) \in \abcs$, a contradiction to Lemma \ref{semi_heights}.
\end{proof}

\section{Insulators}

We say a non-zero integer $n$ is {\bf insulated} if every prime $p \leq P^+(n)$ divides $n$. For example, the primorials $\primorial{n} = \tn{rad}(n!) = \prod_{p \leq n} p$ are insulated. Define the {\bf insulator} of a non-zero integer $n$ as the smallest positive integer $\insulator{n}$ such that $n \cdot \insulator{n}$ is insulated. For example, $\insulator{256256} = \insulator{2^8 \cdot 7 \cdot 11 \cdot 13} = 3 \cdot 5 = 15$. Since \begin{equation}\label{measure_fail}
\primorial{P^+(n)} = \tn{rad}(n) \, \insulator{n},
\end{equation}
the insulator $\insulator{n}$ measures the failure of $\tn{rad}(n)$ to be a primorial. 

\begin{lemma} \label{sandwich} Suppose that $\{ n_k \}_{k = 1}^\infty$ is a sequence of non-zero integers such that $\lim_{k \to \infty} P^+(n_k) = \infty$ and $\insulator{n_k}$ is bounded. Let $\alpha$ and $\beta$ be real numbers such that $0 < \alpha < \log 2$ and $\beta > \log 4$. The inequality 
\begin{equation}\label{lograd_on_smoothness}
\alpha P^+(n_k) < \log \tn{rad}(n_k) < \beta P^+(n_k)
\end{equation}
holds for $k$ sufficiently large.
\end{lemma} 
\begin{proof} Choose a real number $\epsilon > 0$ such that $0 < \alpha/(1-\epsilon) < \log 2$. By equation (\ref{measure_fail}) we have
\begin{equation}
	 \log \tn{rad}(n_k) = \log \primorial{P^+(n_k)} - \log \insulator{n_k}.
\end{equation}
Since $\insulator{n_k}$ is bounded, it follows that
\begin{equation} \label{lograd_on_primorial}
(1-\epsilon) \log \primorial{P^+(n_k)} < \log \tn{rad}(n_k) \leq \log \primorial{P^+(n_k)}
\end{equation}
for $P^+(n_k)$ sufficiently large. This inequality can be written as
\begin{equation}\label{lograd_on_chebyshev}
(1-\epsilon) \theta(P^+(n_k)) < \log \tn{rad}(n_k) \leq \theta(P^+(n_k))
\end{equation} 
in terms of the Chebyshev function $\theta(x) = \sum_{p \leq x} \log p = \log \primorial{x}$. By applying \cite[Corollary 2.10.1]{MR1342300} we obtain
\begin{equation}\label{chebyshev_on_smoothness}
\frac{\alpha}{(1-\epsilon)} P^+(n_k) < \theta(P^+(n_k)) < \beta P^+(n_k)
\end{equation}
when $P^+(n_k)$ is sufficiently large. Combining (\ref{lograd_on_chebyshev}) and (\ref{chebyshev_on_smoothness}) with the hypothesis that $\lim_{n \to \infty} P^+(n_k) = \infty$, we obtain the validity of (\ref{lograd_on_smoothness}) for $k$ sufficiently large.
\end{proof}

\section{Insulators of ABC Solutions}

We now use the language of insulators to study solutions in $\abcs$. Define the {\bf insulator} of a solution $(A,B,C) \in \abcs$ by $\mathcal I(A,B,C) = \insulator{ABC}$.
\begin{conj} \label{insulator_conjecture}The function $\mathcal I : \abcs \to \RR$ defines a semi-height on $\abcs$. 
\end{conj}

\begin{theorem} \label{thm1} Conjecture \ref{weak_xyz} implies Conjecture \ref{insulator_conjecture}.\end{theorem}
\begin{proof} We prove the contrapositive. Let $\{ (A_k, B_k, C_k) \}_{k = 1}^\infty$ be an infinite sequence of distinct elements of
\begin{equation}
\abcs_{\mathcal I}(X) = \{ (A,B,C) \in \abcs \mid \mathcal I(A,B,C) < X \}
\end{equation}
for some real number $X$. Let $\{n_k\}_{k = 1}^\infty$ be given by $n_k = A_k B_k C_k$. 

We show that $\{n_k\}_{k = 1}^\infty$ satisfies the hypotheses of Lemma \ref{sandwich}. Indeed, $\insulator{n_k} = \mathcal I(A_k, B_k, C_k) < X$ for all $k$ since $(A_k, B_k, C_k) \in \abcs_{\mathcal I}(X)$. Also, $\lim_{k \to \infty} P^+(n_k) = \infty$ since $P^+(n_k) = S(A_k, B_k, C_k)$ and $S : \abcs \to \RR$ is a semi-height on $\abcs$ by Lemma \ref{semi_heights}. For any constant $\alpha$ with $0 < \alpha < \log 2$,  Lemma \ref{sandwich} implies
\begin{equation}\label{apply_lemma}
	\alpha P^+(n_k) < \log \tn{rad}(n_k)
\end{equation} 
for $k$ sufficiently large.

To obtain a contradiction to Conjecture \ref{weak_xyz}, we proceed by observing
\begin{align}\label{height_rad}
	\log H(A_k, B_k, C_k) & \geq \tfrac{1}{3} \log \tn{rad}(A_k B_k C_k) = \tfrac{1}{3} \log \tn{rad}(n_k).
\end{align}
Combining inequalities (\ref{apply_lemma}) and (\ref{height_rad}) and letting $\delta = \alpha / 3$ we have 
\begin{align}
	\log H(A_k, B_k, C_k) > \delta \,  P^+(n_k)
\end{align}
 for $k$ sufficiently large. Since $P^+(n_k) = S(A_k, B_k, C_k)$, this implies
\begin{align}
\log H(A_k, B_k, C_k) > \delta \, S(A_k, B_k, C_k)
\end{align}
for all sufficiently large $k$. This contradicts Conjecture \ref{weak_xyz}.
\end{proof}

We now prove the main theorem.
\begin{proof}[Proof of the Main Theorem] Assume the strong form of the XYZ Conjecture holds. Then Proposition \ref{prop1} and Theorem \ref{thm1} imply that $\mathcal I : \abcs \to \RR$ is a semi-height. It follows that all solutions $(A,B,C) \in \abcs$ with a particular insulator $\mathcal I_0$ are contained in the finite set $\abcs_{\mathcal I}(\mathcal I_0 + 1)$. 
\end{proof}


\begin{thebibliography}{bib}
\bibliographystyle{plain}

\bibitem{MR2216774} Enrico Bombieri and Walter Gubler. {\em Heights in Diophantine geometry}, volume 4 of {\em New Mathematical Monographs.} Cambridge University Press, Cambridge, 2006.
\bibitem{MR2780626} Jeffery C. Lagarias and Kannan Soundararajan. Smooth solutions to the $abc$ equation: the $xyz$ conjecture. {\em J. Th\'eor. Nombres. Bordeaux}, 33(1):209--234, 2011.
\bibitem{MR992208} Joseph Oesterl\'e. Nouvelles approches du ``th\'eor\`eme'' de Fermat. {\em Ast\'erisque}, (161-162): Exp. No. 694, 4, 165--186 (1989), 1988. S\'eminaire Bourbaki, Vol. 1987/88.
\bibitem{MR1342300} G\'erald Tenenbaum. {\em Introduction to analytic and probabilistic number theory}, volume 46 of {\em Cambridge Studies in Advanced Mathematics}. Cambridge University Press, Cambridge, 1995. Translated from the second French edition (1995) by C.~B.~Thomas.
\end{thebibliography}
\end{document}